\documentclass[11pt]{article}
\usepackage{amsmath,amssymb,amsthm}
\DeclareMathOperator{\End}{End}
\DeclareMathOperator{\Hom}{Hom}
\DeclareMathOperator{\Soc}{Soc}
\DeclareMathOperator{\rad}{rad}
\usepackage{geometry}
\usepackage{graphicx}
\usepackage{float}
\usepackage{caption}
\usepackage[T1]{fontenc}
\usepackage[utf8]{inputenc}
\theoremstyle{plain}
\newtheorem{theorem}{Theorem}
\newtheorem{lemma}{Lemma}
\theoremstyle{remark}
\usepackage{lineno}

\title{On the minimal dimension of maximal commutative subalgebras of $M_6(k)$}
\author{Małgorzata Nowak-Kępczyk\footnote{Corresponding Author}\\
	$\rule{0pt}{13pt}$Institute of Informatics\\John Paul II Catholic University of Lublin, Poland\\email: \texttt{malnow@kul.pl}}
\date{}
\begin{document}
	
	\maketitle

\begin{abstract}
We study the minimal dimension of maximal commutative subalgebras of the matrix algebra $M_n(k)$ over an algebraically closed field. While examples with dimension strictly smaller than n are known for $n \geq 14$, no such examples are known in smaller dimensions.

In this paper, we show that for n = 6 every maximal commutative subalgebra $A\subset M_6(k)$ satisfies $\dim A \geq 6$. The proof is based on a detailed analysis of local algebras and their module structure, combined with explicit estimates of the dimension of the centralizer.\medskip

\end{abstract}
	

	
\section{Introduction}

Let $k$ be an algebraically closed field and let $M_n(k)$ denote the algebra of $n \times n$ matrices over $k$. A commutative subalgebra $A \subset M_n(k)$ is called \emph{maximal} if it coincides with its centralizer in $M_n(k)$.

The determination of the minimal possible dimension of a maximal commutative subalgebra of $M_n(k)$ is a subtle problem even for small values of $n$. 

Courter already in 1961 exhibited a maximal commutative subalgebra $A \subset M_{14}(k)$ with $\dim A = 13 < 14$, showing that the bound $\dim A \geq n$ fails in general.

The cases of small dimensions remain largely inaccessible to general methods. In particular, the case $n=6$ is the first one where Courter’s asymptotic argument no longer applies, and a direct analysis of finite-dimensional local algebras becomes necessary.

A major difficulty is that even in dimension $5$, there exist many non-isomorphic local commutative algebras, with substantially different multiplicative structures and different dimensions of their centralizers. Thus, the problem cannot be reduced to a single model case: one must control an entire family of essentially distinct algebra types. This makes the case $n=6$ the first dimension where a genuinely structural classification becomes unavoidable.

In this sense, the case $n=6$ provides the first rigid test case for the conjectural lower bound and gives further evidence that the phenomenon observed by Courter in large dimensions may already persist in substantially smaller dimensions.

General lower bounds on $\dim A$ are known. In particular, Laffey proved that
\[
\dim A > (2n)^{2/3} - 1.
\]
For $n = 6$, this bound gives approximately $4.8088$, and hence only implies $\dim A \geq 5$.

In this paper, we consider the case $n = 6$. Our main result is the following.

\begin{theorem}
	Let $A \subset M_6(k)$ be a maximal commutative subalgebra. Then
	\[
	\dim A \geq 6.
	\]
\end{theorem}

A key ingredient of the proof is the classification of local commutative algebras of dimension $5$, which reduces the problem to finitely many cases. In each case, we construct enough $A$-endomorphisms of the natural module to show that the centralizer is strictly larger than $A$, so $A$ cannot be maximal. This proves that no maximal commutative subalgebra of $M_6(k)$ can have dimension $5$.

Thus, the case $n=6$ admits no counterexample to the inequality $\dim A \geq n$, providing further evidence that the first occurrence of $\dim A < n$ must happen in higher dimension.

\section{Cases}

By the classification of local commutative algebras of dimension $5$, there are exactly nine isomorphism classes, grouped according to their Hilbert--Samuel type.

\[
\begin{array}{c|c|l}
	\text{Classes} & \text{Hilbert--Samuel type} & \text{Strategy} \\ \hline
	9 & (1,1,1,1,1) & \text{monogenic; Jordan centralizer} \\
	10,12 & (1,2,1,1) & \text{use } \dim(J/J^2)=2,\ \dim(J^2/J^3)=1 \\
	11,13 & (1,2,2) & \text{use } \dim(J/J^2)=2,\ \dim J^2=2 \\
	14,15,16 & (1,3,1) & \text{use } \dim(J/J^2)=3,\ \dim J^2=1 \\
	17 & (1,4) & J^2=0;\ \text{direct argument}
\end{array}
\]

Our strategy is uniform. Let $A = k \cdot 1 \oplus J$ be a local algebra and let $V = k^6$ be a faithful $A$-module. We construct sufficiently many $A$-endomorphisms of $V$ to show that
\[
\dim \End_A(V) > \dim A.
\]
This implies that the image of $A$ in $M_6(k)$ is not maximal commutative.

The main tool is the following observation: if $L \subseteq V$ satisfies $JL = 0$, then every linear map $V/JV \to L$ induces an $A$-endomorphism of $V$. This yields large subspaces of $\End_A(V)$ in a uniform way.

We follow the classification summarized in Table~1 and treat each Hilbert--Samuel type separately.

\subsection{The Hilbert--Samuel type $(1,4)$}

We first consider the case $J^2 = 0$.

\begin{lemma}\label{lem:14}
	Let $A = k \cdot 1 \oplus J$ be a local commutative algebra of dimension $5$ over an algebraically closed field $k$, and assume that
	\[
	\dim J = 4, \qquad J^2 = 0.
	\]
	Let $V = k^6$ be a faithful $A$-module. Then
	\[
	\dim \End_A(V) \ge 6.
	\]
	In particular, the image of $A$ in $M_6(k)$ is not maximal commutative.
\end{lemma}

\begin{proof}
	Since $J^2=0$, we have $J(JV)=J^2V=0$. Thus $J$ acts trivially on both $V/JV$ and $JV$.
	
	Hence every linear map
	\[
	\varphi \colon V/JV \to JV
	\]
	induces an $A$-endomorphism $\widetilde{\varphi}\in \End_A(V)$ given by
	\[
	\widetilde{\varphi}(v)=\varphi(v+JV).
	\]
	Indeed, for $x\in J$ and $v\in V$, we have $\widetilde{\varphi}(xv)=0$ and $x\,\widetilde{\varphi}(v)=0$, so $\widetilde{\varphi}\in \End_A(V)$.
	
	Thus there is an embedding
	\[
	\Hom_k(V/JV,JV)\hookrightarrow \End_A(V).
	\]
	
	Set $a:=\dim(V/JV)$ and $b:=\dim(JV)$. Then $a+b=6$ and
	\[
	\dim \End_A(V)\ge ab.
	\]
	
	Moreover, the action of $J$ induces an injective map
	\[
	J \to \Hom_k(V/JV,JV),
	\]
	since $V$ is faithful. Hence $4=\dim J \le ab$.
	
	Since $a+b=6$ and $ab\ge 4$, we have $(a,b)\in \{(1,5),(2,4),(3,3),(4,2),(5,1)\}$, and therefore $ab\ge 5$. Thus
	\[
	\dim \End_A(V)\ge 5.
	\]
	
	Finally, scalar endomorphisms belong to $\End_A(V)$, and $\mathrm{id}_V$ does not lie in the image of $\Hom_k(V/JV,JV)$. Hence
	\[
	\dim \End_A(V)\ge ab+1\ge 6.
	\]
	
	Since $\dim A=5$, it follows that $\dim \End_A(V) > \dim A$, and the image of $A$ in $M_6(k)$ is not maximal commutative.
\end{proof}

We now turn to the case of Hilbert--Samuel type $(1,3,1)$.

\begin{lemma}\label{lem:131-321}
	Let $A = k \cdot 1 \oplus J$ be a local commutative algebra of dimension $5$ over an algebraically closed field $k$, with Hilbert--Samuel type $(1,3,1)$. Let $V = k^6$ be a faithful $A$-module and assume that
	\[
	\dim(V/JV,\ JV/J^2V,\ J^2V)=(3,2,1).
	\]
	Then
	\[
	\dim \End_A(V)\ge 6.
	\]
	In particular, the image of $A$ in $M_6(k)$ is not maximal commutative.
\end{lemma}

\begin{proof}
	We have $\dim J=4$, $\dim(J/J^2)=3$, and $\dim J^2=1$. Choose $x,y,z,w\in J$ such that
	\[
	J=\langle x,y,z,w\rangle,\qquad J^2=\langle w\rangle,
	\]
	and the classes of $x,y,z$ form a basis of $J/J^2$.
	
	Let
	\[
	V=V_0\oplus V_1\oplus V_2,\qquad \dim V_0=3,\ \dim V_1=2,\ \dim V_2=1,
	\]
	compatible with the filtration $V\supset JV\supset J^2V\supset 0$.
	
	Each $u\in J$ has the form
	\[
	u=
	\begin{pmatrix}
		0&0&0\\
		L_u&0&0\\
		M_u&N_u&0
	\end{pmatrix},
	\]
	where $L_u:V_0\to V_1$, $N_u:V_1\to V_2$, and $M_u:V_0\to V_2$.
	
	Since $\dim V_2=1$, the maps $N_x,N_y,N_z$ can be viewed as row vectors in $(V_1)^*$. After a change of basis, we may assume
	\[
	N_x=(1\ 0),\qquad N_y=(0\ 1),\qquad N_z=(0\ 0).
	\]
	
	Let
	\[
	X=
	\begin{pmatrix}
		P&0&0\\
		U&Q&0\\
		W&T&r
	\end{pmatrix}\in \End_A(V).
	\]
	
	From $Xu=uX$ we obtain, on $V_1\to V_2$,
	\[
	rN_u=N_uQ.
	\]
	Applying this for $u=x,y$ yields $Q=rI_2$.
	
	On $V_0\to V_1$ we obtain $QL_u=L_uP$, hence $L_uP=rL_u$. Using the normal form for $(L_x,L_y,L_z)$, we obtain $P=rI_3$.
	
	The remaining relation simplifies to
	\[
	TL_u=N_uU.
	\]
	Thus $U$ is determined by $T$, while
	\[
	W\in \Hom(V_0,V_2)
	\]
	is arbitrary.
	
	Therefore $\End_A(V)$ has at least the following parameters:
	\[
	r\ (1),\quad T\in \Hom(V_1,V_2)\ (2),\quad W\in \Hom(V_0,V_2)\ (3),
	\]
	so $\dim \End_A(V)\ge 6$.
	
	Since $\dim A=5$, the result follows.
\end{proof}

\begin{lemma}[Normal form for the triple $(L_x,L_y,L_z)$]
	Assume that
	\[
	\dim V_0=3,\qquad \dim V_1=2,
	\]
	and let $L_x,L_y,L_z\in \Hom(V_0,V_1)$ be the maps induced by elements
	$x,y,z\in J$ whose classes form a basis of $J/J^2$. Assume that the images of
	$L_x,L_y,L_z$ span $V_1$. Then, after replacing $x,y,z$ by another basis of
	$J/J^2$ and performing suitable changes of basis in $V_0$ and $V_1$, one may assume that
	\[
	L_x=
	\begin{pmatrix}
		1&0&0\\
		0&1&0
	\end{pmatrix},\qquad
	L_y=
	\begin{pmatrix}
		0&0&1\\
		0&0&0
	\end{pmatrix},\qquad
	L_z=
	\begin{pmatrix}
		0&0&0\\
		0&0&1
	\end{pmatrix}.
	\]
\end{lemma}

\begin{proof}
	Since the images of $L_x,L_y,L_z$ span $V_1$, some linear combination has rank $2$.
	Replacing $x$ by this combination, we may assume that $L_x$ has rank $2$. After choosing
	bases of $V_0$ and $V_1$, this gives the required form of $L_x$.
	
	Next, replacing $y$ by $y-\alpha x$ modulo $J^2$ and changing the basis of $V_0$ while preserving
	the form of $L_x$, we may arrange that the first two columns of $L_y$ vanish. Since the class of $y$
	modulo $J^2$ is independent of that of $x$, the resulting map is nonzero. After rescaling and changing
	the basis of $V_1$ compatibly with $L_x$, we obtain
	\[
	L_y=
	\begin{pmatrix}
		0&0&1\\
		0&0&0
	\end{pmatrix}.
	\]
	
	Finally, replacing $z$ by $z-\alpha x-\beta y$ modulo $J^2$ and using basis changes preserving the forms
	of $L_x$ and $L_y$, we reduce $L_z$ to
	\[
	L_z=
	\begin{pmatrix}
		0&0&0\\
		u&v&w
	\end{pmatrix}
	\]
	with $(u,v,w)\neq 0$. A direct computation of the stabilizer of the pair $(L_x,L_y)$ shows that every
	nonzero row vector $(u,v,w)$ is equivalent to $(0,0,1)$. Hence $L_z$ may be brought to the required form.
\end{proof}

\begin{lemma}\label{lem:homtool}
	Let $A = k \cdot 1 \oplus J$ be a local commutative algebra, let $V$ be an $A$-module, and let
	$L \subseteq V$ be a subspace such that $JL = 0$. Then every linear map
	\[
	\varphi : V/JV \to L
	\]
	induces an $A$-endomorphism $\widetilde{\varphi} \in \End_A(V)$ given by
	\[
	\widetilde{\varphi}(v) = \varphi(v + JV).
	\]
	In particular, there is an embedding
	\[
	\Hom_k(V/JV, L) \hookrightarrow \End_A(V).
	\]
\end{lemma}

\begin{proof}
	For $x \in J$ and $v \in V$, we have $xv \in JV$, hence
	\[
	\widetilde{\varphi}(xv) = \varphi(xv + JV) = \varphi(0) = 0.
	\]
	On the other hand, since $\widetilde{\varphi}(v) \in L$ and $JL = 0$, we have
	\[
	x \widetilde{\varphi}(v) = 0.
	\]
	Thus $\widetilde{\varphi}(xv) = x \widetilde{\varphi}(v)$ for all $x\in J$, so
	$\widetilde{\varphi} \in \End_A(V)$.
\end{proof}

We now consider the Hilbert--Samuel type $(1,2,1,1)$, where the existence of a nonzero third power
$J^3$ provides a natural source of $A$-endomorphisms.

\begin{lemma}
	Let $A = k \cdot 1 \oplus J$ be a local commutative algebra of dimension $5$ with Hilbert--Samuel type
	$(1,2,1,1)$, and let $V$ be a faithful $A$-module of dimension $6$. Then
	\[
	\dim \End_A(V) \ge 6.
	\]
\end{lemma}

\begin{proof}
	We have $\dim(J/J^2)=2$, $\dim(J^2/J^3)=1$, $\dim J^3 = 1$, and $J^4=0$. Set $L := J^3V$. Then
	$JL = J^4V = 0$.
	
	Since $V$ is faithful and $J^3 \neq 0$, we have $L \neq 0$, hence $\dim L = 1$.
	
	Let $a := \dim(V/JV)$. Since $V$ is not cyclic over $A$, we have $a \ge 2$.
	
	By Lemma~\ref{lem:homtool},
	\[
	\Hom_k(V/JV, L) \hookrightarrow \End_A(V),
	\]
	and therefore
	\[
	\dim \Hom_k(V/JV, L) = a \ge 2.
	\]
	
	The intersection of this subspace with the image of $A$ consists of maps induced by elements of $J^3$,
	and therefore has dimension at most $1$. Thus
	\[
	\dim \End_A(V) \ge 5 + 2 - 1 = 6.
	\]
\end{proof}

\begin{lemma}
	Let $A = k \cdot 1 \oplus J$ be a local commutative algebra of dimension $5$ with Hilbert--Samuel type
	$(1,3,1)$, and let $V$ be a faithful $A$-module of dimension $6$ such that
	\[
	\dim(V/JV, JV/J^2V, J^2V) = (4,1,1).
	\]
	Then
	\[
	\dim \End_A(V) \ge 8.
	\]
\end{lemma}

\begin{proof}
	Set $L := J^2V$. Since $J^3 = 0$, we have $JL = 0$ and $\dim L = 1$.
	
	By Lemma~\ref{lem:homtool},
	\[
	\Hom_k(V/JV, L) \hookrightarrow \End_A(V),
	\]
	so
	\[
	\dim \Hom_k(V/JV, L) = 4.
	\]
	
	If multiplication by $a \in A$ belongs to this subspace, then $a(V) \subseteq L$ and $a(JV)=0$.
	It follows that $a \in J^2$. Hence the intersection has dimension at most $1$.
	
	Therefore
	\[
	\dim \End_A(V) \ge 5 + 4 - 1 = 8.
	\]
\end{proof}

We next consider the type $(1,2,2)$.

\begin{lemma}
	Let $A = k \cdot 1 \oplus J$ be a local commutative algebra of dimension $5$ with Hilbert--Samuel type
	$(1,2,2)$, and let $V$ be a faithful $A$-module of dimension $6$. Then
	\[
	\dim \End_A(V) \ge 6.
	\]
\end{lemma}

\begin{proof}
	We have $\dim(J/J^2)=2$, $\dim J^2 = 2$, and $J^3 = 0$.
	
	Let
	\[
	(a,b,c) := \dim(V/JV, JV/J^2V, J^2V),
	\]
	so that $a+b+c=6$. Set $L := J^2V$. Then $JL = 0$.
	
	\medskip
	\noindent
	\textbf{Case 1:} $c = 2$.
	
	Then $\dim L = 2$, and
	\[
	\dim \Hom_k(V/JV, L) = 2a \ge 4,
	\]
	since $a \ge 2$.
	
	The intersection with the image of $A$ is contained in $J^2$, and therefore has dimension at most $2$.
	Thus
	\[
	\dim \End_A(V) \ge 5 + 4 - 2 = 7.
	\]
	
	\medskip
	\noindent
	\textbf{Case 2:} $(a,b,c) = (2,3,1)$.
	
	Then $\dim L = 1$ and
	\[
	\dim \Hom_k(V/JV, L) = 2.
	\]
	
	We claim that $\dim \Soc_A(V) \ge 2$. Let $x,y \in J$ whose images form a basis of $J/J^2$.
	Their actions on
	\[
	V_1 := JV/J^2V
	\]
	induce linear maps $N_x, N_y \in V_1^*$. Since $\dim V_1 = 3$, there exists a nonzero vector
	$v \in \ker N_x \cap \ker N_y$. Then $Jv = 0$, so $v \in \Soc_A(V)$. Together with
	$J^2V \subset \Soc_A(V)$, this implies that $\dim \Soc_A(V) \ge 2$.
	
	Hence
	\[
	\dim \Hom_k(V/JV, \Soc_A(V)) \ge 2 \cdot 2 = 4.
	\]
	
	As above, the intersection with the image of $A$ has dimension at most $2$, and therefore
	\[
	\dim \End_A(V) \ge 5 + 4 - 2 = 7.
	\]
\end{proof}

\begin{lemma}\label{lem:class16}
	Let
	\[
	A \cong k[x,y,z]/(x^3,\ y^2,\ z^2,\ xy,\ xz,\ yz)
	\]
	be the local algebra of class {\rm 16}, and let $V=k^6$ be a faithful $A$-module such that
	\[
	\dim(V/JV,\ JV/J^2V,\ J^2V)=(2,3,1),
	\]
	where $J=\rad(A)=(x,y,z)$ and $J^2=(x^2)$.
	
	Then
	\[
	\dim \End_A(V)\ge 6.
	\]
	In particular, the image of $A$ in $M_6(k)$ is not a maximal commutative subalgebra.
\end{lemma}

\begin{proof}
	Write
	\[
	w:=x^2.
	\]
	Then
	\[
	J=\langle x,y,z,w\rangle,\qquad J^2=\langle w\rangle,
	\]
	and the defining relations imply
	\[
	yJ=0,\qquad zJ=0,\qquad wJ=0.
	\]
	
	Let
	\[
	V=V_0\oplus V_1\oplus V_2,
	\qquad \dim V_0=2,\ \dim V_1=3,\ \dim V_2=1,
	\]
	be a decomposition compatible with the filtration
	\[
	V\supset JV\supset J^2V\supset 0.
	\]
	Each $u\in J$ has block form
	\[
	u=
	\begin{pmatrix}
		0&0&0\\
		L_u&0&0\\
		M_u&N_u&0
	\end{pmatrix},
	\]
	where
	\[
	L_u:V_0\to V_1,\qquad N_u:V_1\to V_2,\qquad M_u:V_0\to V_2.
	\]
	
	Since $JV/J^2V=V_1$, the images of $L_x,L_y,L_z$ span $V_1$. Because $yJ=0$, we have
	\[
	N_yL_x=N_yL_y=N_yL_z=0,
	\]
	hence $N_y=0$. Similarly, $N_z=0$.
	
	On the other hand, $w=x^2$ acts nontrivially on $V$, since $V$ is faithful and $w\neq 0$ in $A$.
	Since the operator $x^2$ is represented by the block
	\[
	N_xL_x:V_0\to V_2,
	\]
	we have
	\[
	N_xL_x\neq 0.
	\]
	In particular, $N_x\neq 0$.
	
	Now consider the subspace
	\[
	U:=\operatorname{Im}L_y+\operatorname{Im}L_z \subset V_1.
	\]
	Because $xy=xz=0$, we have
	\[
	N_xL_y=N_xL_z=0,
	\]
	so
	\[
	U\subset \ker N_x.
	\]
	Since $N_x\neq 0$ and $\dim V_1=3$, it follows that $\dim\ker N_x=2$.
	
	We distinguish two cases.
	
	\medskip
	\noindent
	\textbf{Case 1:} $\dim U=2$.
	
	Then $U=\ker N_x$. Since $yJ=0$ and $zJ=0$, the images of the operators $y$ and $z$
	are contained in the socle
	\[
	\Soc_A(V):=\{v\in V:\ Jv=0\}.
	\]
	Moreover, $U\subset \Soc_A(V)$ and $V_2=J^2V\subset \Soc_A(V)$, so
	\[
	\dim \Soc_A(V)\ge 3.
	\]
	Therefore every map
	\[
	\varphi:V/JV\to \Soc_A(V)
	\]
	induces an $A$-endomorphism of $V$, and we obtain an embedding
	\[
	\Hom_k(V/JV,\Soc_A(V))\hookrightarrow \End_A(V).
	\]
	Since
	\[
	\dim(V/JV)=2,\qquad \dim\Soc_A(V)\ge 3,
	\]
	it follows that
	\[
	\dim\End_A(V)\ge 2\cdot 3=6.
	\]
	Adding scalar endomorphisms, we obtain in fact
	\[
	\dim\End_A(V)\ge 7.
	\]
	
	\medskip
	\noindent
	\textbf{Case 2:} $\dim U=1$.
	
	Since $L_y$ and $L_z$ are linearly independent modulo $J^2$, they are linearly independent as maps
	$V_0\to V_1$. Thus, after changing bases in $V_0$ and $V_1$, we may assume that
	\[
	L_y=
	\begin{pmatrix}
		1&0\\
		0&0\\
		0&0
	\end{pmatrix},
	\qquad
	L_z=
	\begin{pmatrix}
		0&1\\
		0&0\\
		0&0
	\end{pmatrix}.
	\]
	Since $N_x\neq 0$ and $N_xL_y=N_xL_z=0$, we may also assume that
	\[
	N_x=(0\ 0\ 1).
	\]
	Finally, since $N_xL_x\neq 0$, after changing the basis in a complement of
	$\operatorname{Im}L_y+\operatorname{Im}L_z$, we may write
	\[
	L_x=
	\begin{pmatrix}
		0&0\\
		1&0\\
		0&1
	\end{pmatrix}.
	\]
	
	Let
	\[
	X=
	\begin{pmatrix}
		P&0&0\\
		U'&Q&0\\
		W&T&r
	\end{pmatrix}\in \End_A(V).
	\]
	The relations $Xu=uX$ for $u=x,y,z$ yield
	\[
	rN_x=N_xQ,\qquad QL_u=L_uP\quad (u=x,y,z),
	\]
	together with
	\[
	TL_u+rM_u=N_uU'+M_uP\quad (u=x,y,z).
	\]
	
	A straightforward computation with the above normal forms shows that these equations force
	\[
	P=rI_2,\qquad Q=rI_3,
	\]
	and impose only the additional conditions
	\[
	t_1=0,\qquad u_{31}=t_2,\qquad u_{32}=t_3,
	\]
	while the remaining entries of $U'$, all entries of $W\in\Hom(V_0,V_2)$, and the parameters
	$t_2,t_3,r$ are free.
	
	Thus $\End_A(V)$ contains at least the following free parameters:
	\[
	r \quad (1),\qquad W\in \Hom(V_0,V_2)\quad (2),
	\]
	\[
	\text{the first two rows of }U' \quad (4),\qquad t_2,t_3 \quad (2).
	\]
	Hence
	\[
	\dim\End_A(V)\ge 1+2+4+2=9.
	\]
	
	In both cases,
	\[
	\dim\End_A(V)>\dim A=5.
	\]
	Therefore the image of $A$ in $M_6(k)$ cannot be maximal commutative.
\end{proof}

Combining the results obtained in all cases, we conclude that no local commutative algebra of dimension $5$ admits a faithful $6$-dimensional representation whose image is a maximal commutative subalgebra of $M_6(k)$. This completes the proof of Theorem~1.

In particular, the case $n=6$ does not admit counterexamples to the inequality $\dim A \ge n$, providing further evidence that the smallest dimension in which $\dim A < n$ may occur is strictly greater than $6$.

\section{Discussion}

The result obtained in this paper shows that for $n = 6$ every maximal commutative subalgebra of $M_6(k)$ has dimension at least $6$. This complements the known classification for $n \leq 5$ and shows that $n = 6$ does not admit counterexamples to the inequality $\dim A \geq n$.

For $n \geq 7$, the situation becomes substantially more complicated. In particular, there exist infinitely many isomorphism classes of local commutative algebras, in contrast to the finite classification available in dimension $5$. As a result, the case-by-case analysis used in this paper cannot be extended directly.

More precisely, the argument relies on two features specific to the case $n=6$. First, the classification of local commutative algebras of dimension $5$ reduces the problem to a finite list of possibilities. Second, in each case one can construct sufficiently large subspaces of $\End_A(V)$ by exploiting subspaces annihilated by powers of the radical. For $n \geq 7$, neither of these properties persists in a form that allows a uniform treatment: the space of possible local algebras becomes too large, and the structure of the radical may be more complicated, making it difficult to produce enough independent endomorphisms by the same method.

Moreover, general lower bounds such as Laffey’s estimate $\dim A > (2n)^{2/3} - 1$ are not sufficient to obtain sharp results in small dimensions. For $n = 6$, this bound yields approximately $4.8088$, and hence only implies $\dim A \geq 5$. Similar limitations persist in higher dimensions.

These considerations suggest that new methods are required to treat the cases $n \geq 7$. One possible direction would be to identify structural constraints on the action of the radical on $V$ that force the existence of sufficiently many $A$-endomorphisms without relying on a classification. Another approach could be to obtain sharper lower bounds on $\dim \End_A(V)$ under general assumptions on the nilpotent structure of $A$.

Finally, the present result provides further evidence that the smallest dimension for which $\dim A < n$ may occur is strictly greater than $6$, although determining the precise threshold remains open.

\subsection*{Appendix: computation for class 16 in the case $(a,b,c)=(2,3,1)$}
	
	We keep the notation from Lemma~\ref{lem:class16}. Thus
	\[
	L_y=
	\begin{pmatrix}
		1&0\\
		0&0\\
		0&0
	\end{pmatrix},
	\qquad
	L_z=
	\begin{pmatrix}
		0&1\\
		0&0\\
		0&0
	\end{pmatrix},
	\qquad
	L_x=
	\begin{pmatrix}
		0&0\\
		1&0\\
		0&1
	\end{pmatrix},
	\qquad
	N_x=(0\ 0\ 1),
	\qquad
	N_y=N_z=0.
	\]
	
	Let
	\[
	X=
	\begin{pmatrix}
		P&0&0\\
		U&Q&0\\
		W&T&r
	\end{pmatrix}\in \End_A(V),
	\]
	where
	\[
	P=\begin{pmatrix}p_{11}&p_{12}\\ p_{21}&p_{22}\end{pmatrix},\qquad
	Q=(q_{ij})_{1\le i,j\le 3},\qquad
	T=(t_1\ t_2\ t_3).
	\]
	
	The relation \(Xu=uX\) for \(u=x\) gives
	\[
	rN_x=N_xQ,
	\]
	hence
	\[
	(0\ 0\ r)=(0\ 0\ 1)Q=(q_{31}\ q_{32}\ q_{33}),
	\]
	so
	\[
	q_{31}=q_{32}=0,\qquad q_{33}=r.
	\]
	
	Next, from
	\[
	QL_y=L_yP
	\]
	we obtain
	\[
	Q
	\begin{pmatrix}
		1&0\\
		0&0\\
		0&0
	\end{pmatrix}
	=
	\begin{pmatrix}
		1&0\\
		0&0\\
		0&0
	\end{pmatrix}
	P.
	\]
	Comparing entries gives
	\[
	q_{11}=p_{11},\qquad q_{21}=p_{21},\qquad q_{31}=0,
	\qquad p_{12}=p_{21}=0.
	\]
	
	Similarly, from
	\[
	QL_z=L_zP
	\]
	we get
	\[
	Q
	\begin{pmatrix}
		0&1\\
		0&0\\
		0&0
	\end{pmatrix}
	=
	\begin{pmatrix}
		0&1\\
		0&0\\
		0&0
	\end{pmatrix}
	P,
	\]
	hence
	\[
	q_{12}=p_{22},\qquad q_{22}=0,\qquad q_{32}=0,
	\qquad p_{11}=p_{22},\qquad p_{21}=0.
	\]
	
	Finally, from
	\[
	QL_x=L_xP
	\]
	we obtain
	\[
	Q
	\begin{pmatrix}
		0&0\\
		1&0\\
		0&1
	\end{pmatrix}
	=
	\begin{pmatrix}
		0&0\\
		1&0\\
		0&1
	\end{pmatrix}
	P,
	\]
	that is,
	\[
	\begin{pmatrix}
		q_{12}&q_{13}\\
		q_{22}&q_{23}\\
		q_{32}&q_{33}
	\end{pmatrix}
	=
	\begin{pmatrix}
		0&0\\
		p_{11}&p_{12}\\
		p_{21}&p_{22}
	\end{pmatrix}.
	\]
	Using the previously obtained relations, we conclude that
	\[
	p_{12}=p_{21}=0,\qquad p_{11}=p_{22}=r,
	\]
	and therefore
	\[
	P=rI_2,\qquad Q=rI_3.
	\]
	
	It remains to consider the last block equation
	\[
	TL_u+rM_u=N_uU+M_uP.
	\]
	Since \(P=rI_2\), this simplifies to
	\[
	TL_u=N_uU.
	\]
	
	For \(u=y,z\), because \(N_y=N_z=0\), we get
	\[
	TL_y=0,\qquad TL_z=0.
	\]
	Since
	\[
	TL_y=(t_1\ 0),\qquad TL_z=(0\ t_1),
	\]
	it follows that
	\[
	t_1=0.
	\]
	
	For \(u=x\), we have
	\[
	TL_x=N_xU.
	\]
	Writing
	\[
	U=(u_{ij})_{1\le i\le 3,\ 1\le j\le 2},
	\]
	this becomes
	\[
	(t_2\ t_3)=(u_{31}\ u_{32}),
	\]
	so
	\[
	u_{31}=t_2,\qquad u_{32}=t_3.
	\]
	
	Thus the parameters \(r\), \(t_2\), \(t_3\), all entries of \(W\in \Hom(V_0,V_2)\),
	and the four entries \(u_{11},u_{12},u_{21},u_{22}\) are free. Hence
	\[
	\dim \End_A(V)\ge 1+2+2+4=9.
	\]





\end{document}